\documentclass[reqno]{amsart}
\usepackage{amssymb,mathtools,mathrsfs,lineno,anysize}%
\marginsize{2.5cm}{2.5cm}{2.5cm}{2.5cm}

\usepackage[colorlinks=true,citecolor=blue,linkcolor=green,urlcolor=red]{hyperref}

\allowdisplaybreaks[4]
\theoremstyle{plain}
\newtheorem{thm}{Theorem}
\newtheorem{lem}{Lemma}
\newtheorem{prop}{Proposition}
\newtheorem{cor}{Corollary}

\theoremstyle{remark}
\newtheorem{rem}{Remark}

\DeclareMathOperator{\td}{d\!}
\DeclareMathOperator{\te}{e}

\DeclareMathOperator{\diag}{diag}

\begin{document}
\title[Hurwitz zeta functions and Gauss hypergeometric functions]
{Expressing a difference between two Hurwitz zeta functions in terms of the Gauss hypergeometric functions}

\author[F. Qi]{Feng Qi}
\address{School of Mathematics and Physics, Hulunbuir University, Hulunbuir 021008, Inner Mongolia, China;
17709 Sabal Court, University Village, Dallas, TX 75252-8025, USA}
\email{\href{mailto:qifeng618@gmail.com}{qifeng618@gmail.com}}
\urladdr{\url{https://orcid.org/0000-0001-6239-2968}}



\begin{abstract}
In the paper, the author expresses the difference $2^m[\zeta(-m,\lambda z)-\zeta(-m,\lambda z+\mu)]$ for $m\in\mathbb{N}_0$ and $\Re(z)>0$ with $\lambda,\mu>0$ in terms of the functions $\Gamma(j+1){\,}_2F_1(-j,-z;1;2)$ for $0\le j\le m$ and $z\in\mathbb{C}$ in the form of matrix equations, where $\Gamma(z)$, $\zeta(z,\alpha)$, and ${}_2F_1(a,b;c;z)$ stand for the classical Euler gamma function, the Hurwitz zeta function, and the Gauss hypergeometric function, respectively. The case of non-integer $t$ in $2^t[\zeta(-t,\lambda z)-\zeta(-t,\lambda z+\mu)]$ is left open. The coefficients are characterized by matrix formulas and by an explicit recursion; closed-form expressions for them are not obtained. This problem originates from the Landau level quantization in solid state materials.
\end{abstract}

\thanks{This paper was typeset using \AmS-\LaTeX}

\maketitle

\section{Preliminaries}

The Hurwitz zeta function $\zeta(z,\alpha)$ is defined by
\[
\zeta(z,\alpha)=\sum_{n=0}^{\infty}\frac{1}{(n+\alpha)^z}, \quad \Re(z)>1
\]
for $\alpha\not=0,-1,-2,-3,\dotsc$. 
It has a meromorphic continuation in the $z$-plane and its only singularity in $\mathbb{C}$ is a simple pole at $z=1$ with residue $1$. As a function of $\alpha$, with $z\ne1$ fixed, the function $\zeta(z,\alpha)$ is analytic in the half-plane $\Re(\alpha)>0$. The Riemann zeta function $\zeta(z)$ is the special case $\zeta(z,1)$. For more details, please refer to~\cite[p.~607]{NIST-HB-2010} and~\cite[pp.~75--76]{Temme-96-book}.
\par
The Dirichlet eta function $\eta(z)$ and the Riemann zeta function $\zeta(z)$ have the relation
\begin{equation*}
\eta(z)=\bigl(1-2^{1-z}\bigr)\zeta(z).
\end{equation*}
See~\cite{TRMI-2025-199.tex, CMP6528.tex, Mon-Eta-Ratio.tex, log-secant-norm-tail.tex} and closely related references therein.
\par
The classical Euler gamma function $\Gamma(z)$ can be defined~\cite[Chapter~3]{Temme-96-book} by
\begin{equation*}
\Gamma(z)=\lim_{n\to\infty}\frac{n!n^z}{\prod_{k=0}^n(z+k)}, \quad z\in\mathbb{C}\setminus\{0,-1,-2,\dotsc\}.
\end{equation*}
In terms of the Pochhammer symbol, also known as the rising factorial,
\begin{equation}\label{rising-Factorial}
(z)_n=\frac{\Gamma(z+n)}{\Gamma(z)}
=\prod_{\ell=0}^{n-1}(z+\ell)
=
\begin{cases}
z(z+1)\dotsm(z+n-1), & n\in\mathbb{N}\\
1, & n=0
\end{cases}
\end{equation}
for $z\in\mathbb{C}$, the Gauss hypergeometric function ${}_2F_1(a,b;c;z)$ is defined by
\begin{equation}\label{hypergeom-f}
{\,}_2F_1(a,b;c;z)=\sum_{k=0}^\infty\frac{(a)_k(b)_k} {(c)_k}\frac{z^k}{k!}
\end{equation}
for $a,b\in\mathbb{C}$ and $c\in\mathbb{C}\setminus\{0,-1,-2,\dotsc\}$.
When $a\in\{0,-1,-2,\dotsc\}$ or $b\in\{0,-1,-2,\dotsc\}$, the series in~\eqref{hypergeom-f} is finite and is a polynomial of $z\in\mathbb{C}$. When $a,b\not\in\{0,-1,-2,\dotsc\}$, the infinite series~\eqref{hypergeom-f} converges in the unit disc $|z|<1$ and, on the circle $|z|=1$, is
\begin{enumerate}
\item
divergent for $\Re(a+b-c)\ge1$,
\item
absolutely convergent for $\Re(a+b-c)<0$,
\item
conditionally convergent for $0\le\Re(a+b-c)<1$, the point $z=1$ being excluded.
\end{enumerate}
For $c\in\{0,-1,-2,\dotsc\}$, if $a\in\{-1,-2,\dotsc\}$ such that $a>c$ or $b\in\{-1,-2,\dotsc\}$ such that $b>c$, the Gauss hypergeometric function ${}_2F_1(a,b;c;z)$ is also defined. In a word, the function ${\,}_2F_1(a,b;c;z)$ is a single-valued analytic function in the $z$-plane with a cut from $1$ to $+\infty$ along the real axis; the power series in~\eqref{hypergeom-f} is a single-valued branch, the branch which is equal to $1$ when $z=0$, representing it in $|z|<1$.
For more information on the Gauss hypergeometric function ${}_2F_1(a,b;c;z)$, please refer to~\cite[Chapter~15]{abram}, \cite[Chapter~5]{Temme-96-book}, and~\cite[Chapter~4]{Wang-Guo-1989SF}.
\par
Paired with the rising factorial $(z)_n$ in~\eqref{rising-Factorial}, the falling factorial $\langle z\rangle_n$ was defined in~\cite[p.~6]{Comtet-Combinatorics-74} by
\begin{equation}\label{Fall-Factorial-Dfn-Eq}
\langle z\rangle_n=\frac{\Gamma(z+1)}{\Gamma(z-n+1)}
=
\prod_{k=0}^{n-1}(z-k)=
\begin{cases}
z(z-1)\dotsm(z-n+1), & n\in\mathbb{N};\\
1,& n=0.
\end{cases}
\end{equation}
\par
It is well known~\cite[p.~3]{Temme-96-book} that the Bernoulli polynomials $B_{j}(t)$ are generated by
\begin{equation*}
\frac{z\te^{t z}}{\te^z-1}=\sum_{j=0}^{\infty}B_j(t)\frac{z^j}{j!}, \quad |z|<2\pi
\end{equation*}
and that the sequence $B_j=B_j(0)$ for $j\in\mathbb{N}_0$ stands for the Bernoulli numbers.
\par
The Stirling numbers of the first and second kinds $s(k,j)$ and $S(k,j)$ are basic and fundamental notions in combinatorial number theory.
They can be analytically generated by
\begin{equation*}
\biggl[\frac{\ln(1+z)}{z}\biggr]^j=\sum_{k=0}^\infty \frac{s(k+j,j)}{\binom{k+j}{j}}\frac{z^{k}}{k!}, \quad |z|<1
\end{equation*}
and
\begin{equation*}
\biggl(\frac{\te^z-1}{z}\biggr)^j=\sum_{k=0}^\infty \frac{S(k+j,j)}{\binom{k+j}{j}} \frac{z^{k}}{k!}, \quad j\ge0,
\end{equation*}
respectively; see~\cite[pp.~46 and~56]{Mansour-Schork-B2016} and~\cite[pp.~132 and~168]{Quaintance-Gould-Stirling-B}.

\section{Inverse of lower triangular matrices}
In this section, we derive the inverse of lower triangular matrices. In the next section, we will use the inverse to establish our main results.

\begin{prop}\label{Im-Inverse-lem}
Let $I_m$ be the $m\times m$ identity matrix, let $L_m$ be a strictly lower triangular $m\times m$ matrix, and let $D_m$ be a diagonal $m\times m$ matrix with non-zero diagonal entries. Then the inverses
\begin{equation}\label{Im-Inverse-lem-eq}
(I_m+L_m)^{-1}=I_m+\sum_{k=1}^{m-1}(-1)^kL_m^k
\end{equation}
and
\begin{equation}\label{beta-matrix-inv-D-eq}
(D_m+L_m)^{-1}=\Biggl[I_m+\sum_{k=1}^{m-1}(-1)^k\bigl(D_m^{-1}L_m\bigr)^k\Biggr]D_m^{-1}
\end{equation}
are valid.
\end{prop}

\begin{proof}
It is common knowledge that $L_m^m=O_m$, a null $m\times m$ matrix whose elements are all zeros, in linear algebra or the theory of matrices. Then
\begin{equation*}
(I_m+L_m)\Biggl[I_m+\sum_{k=1}^{m-1}(-1)^kL_m^k\Biggr]=I_m.
\end{equation*}
Accordingly, the formula~\eqref{Im-Inverse-lem-eq} is valid.
\par
Making use of the equation
\begin{equation*}
D_m+L_m=D_m\bigl(I_m+D_m^{-1}L_m\bigr)
\end{equation*}
and the inverse in~\eqref{Im-Inverse-lem-eq}, we conclude the inverse~\eqref{beta-matrix-inv-D-eq} immediately.
The proof of Proposition~\ref{Im-Inverse-lem} is thus complete.
\end{proof}

\section{Main results}
Let
\begin{equation}\label{F(m-n)}
F(m,n)=2^m\biggl[\zeta\biggl(-m,\frac{1+n}{2}\biggr)-\zeta\biggl(-m,\frac{2+n}{2}\biggr)\biggr]
\end{equation}
and
\begin{equation}\label{G(m-n)}
G(m,n)=\Gamma(m+1){\,}_2F_1(-m,-n;1;2)
\end{equation}
for $m,n\in\mathbb{N}_0$.
When studying the Landau level quantization in solid state materials, Xuan-Yu Long at the University of Utah needs a linear combination for expressing $F(m,n)$ in terms of $G(m,n)$. In other words, he needs to determine the scalars $a_{m,j}$ such that the linear combination
\begin{equation}\label{Long-Propb}
F(m,n)=\sum_{j=0}^m a_{m,j} G(j,n)
\end{equation}
is valid for $m,n\in\mathbb{N}_0$.
\par
In this section, we will give several matrix formulas for the matrix $\begin{pmatrix}a_{i,j}\end{pmatrix}_{0\le i,j\le m}$. These formulas, together with a recursion for computing the entries $a_{i,j}$ (Corollary~\ref{recursion-cor}), give a computable description of the matrix; a closed-form expression for $a_{i,j}$ remains open (see Remark~\ref{open-aij-rem}).
\par
It is easy to see that the sequences $F(m,n)$ and $G(m,n)$ for $m,n\in\mathbb{N}_0$ in~\eqref{F(m-n)} and~\eqref{G(m-n)} can be extended to $F(m,z)$ and $G(m,z)$ for $m\in\mathbb{N}_0$ and $z\in(-1,\infty)$.

\begin{lem}\label{unique-lem}
Let $m\in\mathbb{N}_0$, let $S\subset\mathbb{R}$ be an infinite set, and let $p_0,p_1,\dotsc,p_m$ be polynomials with $\deg p_j=j$. If scalars $c_0,c_1,\dotsc,c_m$ satisfy
\begin{equation*}
\sum_{j=0}^{m}c_jp_j(z)=0 \quad\text{for all } z\in S,
\end{equation*}
then $c_0=c_1=\dotsb=c_m=0$. Consequently, if a function $H$ satisfies $H(z)=\sum_{j=0}^{m}a_jp_j(z)$ for all $z\in S$, then the scalars $a_0,\dotsc,a_m$ are uniquely determined.
\end{lem}

\begin{proof}
The polynomial $P(z)=\sum_{j=0}^{m}c_jp_j(z)$ has degree at most $m$ and vanishes on the infinite set $S$, so $P\equiv0$. If some $c_j\neq0$, let $J$ be the largest such index. Then $\deg P=J$, which contradicts $P\equiv0$. The uniqueness statement follows by applying this to the difference of two representations of $H$.
\end{proof}

\begin{thm}\label{FG(m-x)-matrix-Rel-thm}
For $m\in\mathbb{N}_0$ and $z\in(-1,\infty)$, the matrix equation
\begin{equation}\label{FG(m-x)-matrix-Relation}
\begin{pmatrix}
F(0,z)\\ F(1,z)\\ F(2,z)\\ F(3,z)\\ \vdots\\ F(m-1,z) \\ F(m,z)
\end{pmatrix}
=
\mathcal{A}_{m+1}\mathcal{B}_{m+1}^{-1}
\begin{pmatrix}
G(0,z)\\ G(1,z)\\ G(2,z)\\ G(3,z)\\ \vdots\\ G(m-1,z) \\ G(m,z)
\end{pmatrix},
\end{equation}
or, equivalently,
\begin{equation}\label{aij-mstrix}
\begin{pmatrix}a_{i,j}\end{pmatrix}_{0\le i,j\le m}=\mathcal{A}_{m+1}\mathcal{B}_{m+1}^{-1},
\end{equation}
is uniquely valid, where
\begin{gather}\label{Matrix-A(m+1)-dfn}
\mathcal{A}_{m+1}=\begin{pmatrix}\alpha_{i,j}\end{pmatrix}_{0\le i,j\le m}
=\begin{pmatrix}\displaystyle\frac{2^{i}}{i+1}\sum_{k=j}^{i}\binom{i+1}{k+1} \binom{k+1}{j}\frac{2^{k-j+1}-1}{2^{k+1}}B_{i-k}\end{pmatrix}_{0\le i,j\le m}, \\
\mathcal{B}_{m+1}=
\begin{pmatrix}
\beta_{i,j}
\end{pmatrix}_{0\le i,j\le m}
=\begin{pmatrix}\displaystyle\sum_{k=j}^i2^k(i-k)!\binom{i}{k}^2 s(k,j)\end{pmatrix}_{0\le i,j\le m}, \label{Matrix-B(m+1)-dfn}
\end{gather}
and the notations $B_{i-k}$ and $s(k,j)$ stand for the Bernoulli numbers and the Stirling numbers of the first kind, respectively.
\end{thm}

\begin{proof}
In~\cite[p.~608 and p.~588]{NIST-HB-2010}, we find
\begin{equation}\label{zeta-bernoulli-no}
\zeta(-m,a)=-\frac{B_{m+1}(a)}{m+1}, \quad m\in\mathbb{N}_0
\end{equation}
and
\begin{equation}\label{p.588NIST-HB-2010}
B_n(z)=\sum_{k=0}^{n}\binom{n}{k}B_kz^{n-k}, \quad n\in\mathbb{N}_0.
\end{equation}
Thus, we arrive at
\begin{align*}
F(m,z)&=\frac{2^m}{m+1}\biggl[B_{m+1}\biggl(\frac{2+z}{2}\biggr)-B_{m+1}\biggl(\frac{1+z}{2}\biggr)\biggr]\\
&=\frac{1}{m+1}\sum_{k=0}^{m+1}\binom{m+1}{k} B_k2^{k-1}\bigl[(2+z)^{m+1-k} -(1+z)^{m+1-k}\bigr]\\
&=\frac{1}{m+1}\sum_{k=0}^{m+1}\binom{m+1}{k} B_k2^{k-1} \sum_{j=0}^{m+1-k}\binom{m+1-k}{j}\bigl(2^{m+1-k-j}-1\bigr)z^j\\
&=\frac{1}{m+1}\sum_{\ell=0}^{m+1}\binom{m+1}{\ell} B_{m-\ell+1}2^{m-\ell} \sum_{j=0}^{\ell}\binom{\ell}{j}\bigl(2^{\ell-j}-1\bigr)z^j\\
&=\frac{2^{m}}{m+1}\sum_{j=0}^{m+1}\Biggl[\sum_{k=j}^{m+1}\binom{m+1}{k} \binom{k}{j}\frac{2^{k-j}-1}{2^{k}}B_{m-k+1}\Biggr]z^j
\end{align*}
for $m\in\mathbb{N}_0$ and $z\in(-1,\infty)$. As a result, we obtain
\begin{equation}\label{F(m-z)-z-polyn}
F(m,z)=\frac{2^{m}}{m+1}\sum_{j=0}^{m}\Biggl[\sum_{k=j}^{m}\binom{m+1}{k+1} \binom{k+1}{j}\frac{2^{k-j+1}-1}{2^{k+1}}B_{m-k}\Biggr]z^j\\
=\sum_{j=0}^{m}\alpha_{m,j}z^j
\end{equation}
for $m\in\mathbb{N}_0$ and $z\in(-1,\infty)$. The function $F(m,z)$ is a polynomial of degree $m$ in the variable $z$.
\par
By the definition~\eqref{hypergeom-f} for $a\in\{0,-1,-2,\dotsc\}$ or $b\in\{0,-1,-2,\dotsc\}$, we have
\begin{equation*}
{\,}_2F_1(-m,-z;1;2)=\sum_{k=0}^m\frac{(-m)_k2^k} {(k!)^2}(-z)_k
=\sum_{k=0}^m\frac{\langle m\rangle_k2^k} {(k!)^2}\sum_{j=0}^ks(k,j)z^j
=\sum_{j=0}^m\Biggl[\sum_{k=j}^m\frac{\langle m\rangle_k2^k} {(k!)^2}s(k,j)\Biggr]z^j
\end{equation*}
for $m\in\mathbb{N}_0$ and $z\in\mathbb{R}$, where we used the expansion
\begin{equation}\label{p6Comtet-Combinatorics-74}
\langle z\rangle_k=\sum_{j=0}^k s(k,j)z^j
\end{equation}
in~\cite[p.~29]{Jr.Hall-B-1986} and the relation
\begin{equation}\label{falling2rising}
(z)_k=(-1)^k\langle-z\rangle_k,
\end{equation}
which can be directly derived from comparing~\eqref{rising-Factorial} and~\eqref{Fall-Factorial-Dfn-Eq}.
Further making use of the first equality in~\eqref{Fall-Factorial-Dfn-Eq}, we obtain
\begin{equation}\label{G(m-z)-z-polyn}
G(m,z)=m!\sum_{j=0}^m\Biggl[\sum_{k=j}^m\frac{\langle m\rangle_k2^k} {(k!)^2}s(k,j)\Biggr]z^j
=\sum_{j=0}^m\Biggl[\sum_{k=j}^m2^k(m-k)!\binom{m}{k}^2 s(k,j)\Biggr]z^j
=\sum_{j=0}^m\beta_{m,j}z^j
\end{equation}
for $m\in\mathbb{N}_0$ and $z\in\mathbb{R}$. The function $G(m,z)$ is a polynomial of degree $m$ in $z$.
\par
The formulas~\eqref{F(m-z)-z-polyn} and~\eqref{G(m-z)-z-polyn} can be rearranged as matrix equations
\begin{equation*}
\begin{pmatrix}
F(0,z)\\ F(1,z)\\ F(2,z)\\ F(3,z)\\ \vdots\\ F(m-1,z) \\ F(m,z)
\end{pmatrix}
=\mathcal{A}_{m+1}\begin{pmatrix}
1\\ z\\ z^2\\ z^3\\ \vdots\\ z^{m-1}\\ z^m
\end{pmatrix}
\end{equation*}
and
\begin{equation*}
\begin{pmatrix}
G(0,z)\\ G(1,z)\\ G(2,z)\\ G(3,z)\\ \vdots\\ G(m-1,z) \\ G(m,z)
\end{pmatrix}
=\mathcal{B}_{m+1}\begin{pmatrix}
1\\ z\\ z^2\\ z^3\\ \vdots\\ z^{m-1}\\ z^m
\end{pmatrix}
\end{equation*}
for $m\in\mathbb{N}_0$ and $z\in(-1,\infty)$, respectively.
Because $\beta_{i,i}=2^i$ for $0\le i\le m\in\mathbb{N}_0$, the $(m+1)\times(m+1)$ square matrix $\mathcal{B}_{m+1}$ is invertible. Consequently, we acquire the matrix equation~\eqref{FG(m-x)-matrix-Relation}.
By~\eqref{G(m-z)-z-polyn}, $G(j,\cdot)$ is a polynomial of degree exactly $j$, since $\beta_{j,j}=2^j\neq0$. Hence the coefficients $a_{i,j}$ are unique by Lemma~\ref{unique-lem} with $p_j=G(j,\cdot)$ and $S=(-1,\infty)$ or $S=\mathbb{N}_0$.
The proof of Theorem~\ref{FG(m-x)-matrix-Rel-thm} is complete.
\end{proof}

In Theorem~\ref{FG(m-x)-matrix-Rel-thm}, both $F(m,z)$ and $G(m,z)$ are expanded in powers of $z$, which produces the matrices $\mathcal{A}_{m+1}$ and $\mathcal{B}_{m+1}$. In the next theorem, both are expanded in powers of $z+1$, which produces a different pair $\mathfrak{A}_{m+1}$ and $\mathfrak{B}_{m+1}$. The expansion of $F(m,z)$ in powers of $z+1$ has a simpler coefficient formula, since the factor $2^{k-j+1}-1$ in~\eqref{Matrix-A(m+1)-dfn} does not appear, while the expansion of $G(m,z)$ uses the shifted Stirling numbers $s(k+1,j+1)$. The two factorizations are different, but their products coincide, as shown by uniqueness of the coefficients $a_{m,j}$.

\begin{thm}\label{frak-AB-matrix-thm}
For $m\in\mathbb{N}_0$, let
\begin{equation}\label{lambda-matrix}
\mathfrak{A}_{m+1}=
\begin{pmatrix}
\lambda_{i,j}
\end{pmatrix}_{0\le i,j\le m}
=\begin{pmatrix}\displaystyle
\sum_{k=0}^{i-j} \binom{i}{k}\frac{2^{k-1}B_k}{i-k+1} \binom{i-k+1}{j}
\end{pmatrix}_{0\le i,j\le m}
\end{equation}
and
\begin{equation}\label{mu-matrix}
\mathfrak{B}_{m+1}=
\begin{pmatrix}
\mu_{i,j}
\end{pmatrix}_{0\le i,j\le m}
=\begin{pmatrix}\displaystyle
\sum_{k=j}^i2^k(i-k)!\binom{i}{k}^2s(k+1,j+1)
\end{pmatrix}_{0\le i,j\le m}.
\end{equation}
Then the matrix equation
\begin{equation}\label{frak-AB-matrix}
\begin{pmatrix}
F(0,z)\\ F(1,z)\\ F(2,z)\\ F(3,z)\\ \vdots\\ F(m-1,z) \\ F(m,z)
\end{pmatrix}
=\mathfrak{A}_{m+1}\mathfrak{B}_{m+1}^{-1}
\begin{pmatrix}
G(0,z)\\ G(1,z)\\ G(2,z)\\ G(3,z)\\ \vdots\\ G(m-1,z) \\ G(m,z)
\end{pmatrix}
\end{equation}
or, equivalently,
\begin{equation}\label{frakaij-mstrix}
\begin{pmatrix}a_{i,j}\end{pmatrix}_{0\le i,j\le m}=\mathfrak{A}_{m+1}\mathfrak{B}_{m+1}^{-1},
\end{equation}
is uniquely valid for $m\in\mathbb{N}_0$ and $z\in(-1,\infty)$.
\end{thm}

\begin{proof}
Utilizing \eqref{zeta-bernoulli-no}, the formula
\begin{equation}\label{p805Entry23.1.11abram}
\int_{a}^{x} B_n(t)\td t=\frac{B_{n+1}(x)-B_{n+1}(a)}{n+1}, \quad n\in\mathbb{N}_0, \quad a,x\in\mathbb{R}
\end{equation}
in~\cite[p.~805, Entry~23.1.11]{abram}, and~\eqref{p.588NIST-HB-2010}, we arrive at
\begin{align*}
F(m,z)&=\frac{2^m}{m+1}\biggl[B_{m+1}\biggl(\frac{2+z}{2}\biggr)-B_{m+1}\biggl(\frac{1+z}{2}\biggr)\biggr]\\
&=2^m\int_{(1+z)/2}^{(2+z)/2}B_m(t)\td t\\
&=2^m\int_{(1+z)/2}^{(2+z)/2} \sum_{k=0}^{m}\binom{m}{k}B_k t^{m-k}\td t\\
&=2^m\sum_{k=0}^{m}\binom{m}{k}\frac{B_k}{m-k+1} \biggl[\biggl(\frac{2+z}{2}\biggr)^{m-k+1}-\biggl(\frac{1+z}{2}\biggr)^{m-k+1}\biggr]\\
&=\sum_{k=0}^{m} \binom{m}{k}\frac{2^{k-1}B_k}{m-k+1} \bigl[(2+z)^{m-k+1}-(1+z)^{m-k+1}\bigr]\\
&=\sum_{k=0}^{m} \binom{m}{k}\frac{2^{k-1}B_k}{m-k+1} \sum_{j=0}^{m-k}\binom{m-k+1}{j}(z+1)^j\\
&=\sum_{j=0}^{m}\Biggl[\sum_{k=0}^{m-j} \binom{m}{k}\frac{2^{k-1}B_k}{m-k+1} \binom{m-k+1}{j}\Biggr](z+1)^j
\end{align*}
for $m\in\mathbb{N}_0$ and $z\in(-1,\infty)$. In other words,
\begin{equation}\label{frak-A-matrix}
\begin{pmatrix}
F(0,z)\\ F(1,z)\\ F(2,z)\\ F(3,z)\\ \vdots\\ F(m-1,z) \\ F(m,z)
\end{pmatrix}
=\mathfrak{A}_{m+1}
\begin{pmatrix}
1\\ z+1\\ (z+1)^2\\ (z+1)^3\\ \vdots\\ (z+1)^{m-1}\\ (z+1)^m
\end{pmatrix}, \quad m\in\mathbb{N}_0, \quad z\in(-1,\infty).
\end{equation}
\par
By the definition in~\eqref{rising-Factorial}, we acquire
\[
(z)_k=\frac{\Gamma(z+k)}{\Gamma(z)}=\frac{\Gamma(z+k-1+1)}{\Gamma(z-1+1)}
=\frac{(z+k-1)\Gamma(z+k-1)}{(z-1)\Gamma(z-1)}
=\frac{z+k-1}{z-1}(z-1)_k.
\]
Accordingly, utilizing~\eqref{p6Comtet-Combinatorics-74} and~\eqref{falling2rising}, we obtain
\begin{align*}
(-z)_k&=\frac{-z-1+k}{-z-1}(-z-1)_k\\
&=\frac{z+1-k}{z+1}(-1)^k\langle z+1\rangle_k\\
&=(-1)^k\frac{z+1-k}{z+1}\sum_{j=0}^k s(k,j)(z+1)^j\\
&=(-1)^k\Biggl[\sum_{j=0}^k s(k,j)(z+1)^j-k\sum_{j=0}^k s(k,j)(z+1)^{j-1}\Biggr]\\
&=(-1)^k\Biggl[\sum_{j=0}^k s(k,j)(z+1)^j-k\sum_{j=0}^{k-1} s(k,j+1)(z+1)^{j}\Biggr]\\
&=(-1)^k\Biggl[(z+1)^k+\sum_{j=0}^{k-1}[s(k,j)-k s(k,j+1)](z+1)^{j}\Biggr]\\
&=(-1)^k\Biggl[(z+1)^k+\sum_{j=0}^{k-1}s(k+1,j+1)(z+1)^{j}\Biggr]\\
&=(-1)^k\sum_{j=0}^{k}s(k+1,j+1)(z+1)^{j}.
\end{align*}
Hence, by the definition~\eqref{hypergeom-f} for $a\in\{0,-1,-2,\dotsc\}$ or $b\in\{0,-1,-2,\dotsc\}$, we obtain
\begin{align*}
{\,}_2F_1(-m,-z;1;2)&=\sum_{k=0}^m\frac{(-m)_k2^k} {(k!)^2}(-z)_k\\
&=\sum_{k=0}^m\frac{\langle m\rangle_k2^k} {(k!)^2}\sum_{j=0}^{k}s(k+1,j+1)(z+1)^{j}\\
&=\sum_{j=0}^{m}\Biggl[\sum_{k=j}^m\frac{\langle m\rangle_k2^k} {(k!)^2}s(k+1,j+1)\Biggr](z+1)^{j}
\end{align*}
for $m\in\mathbb{N}_0$ and $z\in(-1,\infty)$. As a result, we acquire
\begin{align*}
G(m,z)&=m!\sum_{j=0}^{m}\Biggl[\sum_{k=j}^m\frac{\langle m\rangle_k2^k} {(k!)^2}s(k+1,j+1)\Biggr](z+1)^{j}\\
&=\sum_{j=0}^{m}\Biggl[\sum_{k=j}^m2^k(m-k)!\binom{m}{k}^2s(k+1,j+1)\Biggr](z+1)^{j}
\end{align*}
for $m\in\mathbb{N}_0$ and $z\in(-1,\infty)$. In other words,
\begin{equation}\label{frak-B-matrix}
\begin{pmatrix}
G(0,z)\\ G(1,z)\\ G(2,z)\\ G(3,z)\\ \vdots\\ G(m-1,z) \\ G(m,z)
\end{pmatrix}
=\mathfrak{B}_{m+1}\begin{pmatrix}
1\\ z+1\\ (z+1)^2\\ (z+1)^3\\ \vdots\\ (z+1)^{m-1}\\ (z+1)^m
\end{pmatrix}, \quad m\in\mathbb{N}_0, \quad z\in(-1,\infty).
\end{equation}
\par
Since $\mu_{i,i}=2^i$ for $0\le i\le m$, the $(m+1)\times(m+1)$ square matrix $\mathfrak{B}_{m+1}$ is invertible. Combining~\eqref{frak-A-matrix} and~\eqref{frak-B-matrix} leads to~\eqref{frak-AB-matrix}. Since $G(j,\cdot)$ is a polynomial of degree exactly $j$ by~\eqref{G(m-z)-z-polyn}, the coefficients $a_{i,j}$ are unique by Lemma~\ref{unique-lem} with $p_j=G(j,\cdot)$ and $S=(-1,\infty)$ or $S=\mathbb{N}_0$. The proof of Theorem~\ref{frak-AB-matrix-thm} is complete.
\end{proof}

\begin{thm}\label{FG(m-x)-matrix-Final-Thm}
For $m\in\mathbb{N}_0$ and $z\in(-1,\infty)$, we have
\begin{equation}\label{FG(m-x)-matrix-Final}
\begin{pmatrix}
F(0,z)\\ F(1,z)\\ F(2,z)\\ F(3,z)\\ \vdots\\ F(m-1,z) \\ F(m,z)
\end{pmatrix}
=
\mathcal{A}_{m+1}\Biggl[I_{m+1}+\sum_{k=1}^{m}(-1)^k\bigl(\mathcal{D}_{m+1}^{-1} \mathcal{L}_{m+1}\bigr)^k\Biggr]\mathcal{D}_{m+1}^{-1}
\begin{pmatrix}
G(0,z)\\ G(1,z)\\ G(2,z)\\ G(3,z)\\ \vdots\\ G(m-1,z) \\ G(m,z)
\end{pmatrix}
\end{equation}
and
\begin{equation}\label{FG(m-x)-matrix-frak}
\begin{pmatrix}
F(0,z)\\ F(1,z)\\ F(2,z)\\ F(3,z)\\ \vdots\\ F(m-1,z) \\ F(m,z)
\end{pmatrix}
=
\mathfrak{A}_{m+1}\Biggl[I_{m+1}+\sum_{k=1}^{m}(-1)^k\bigl(\mathfrak{D}_{m+1}^{-1} \mathfrak{L}_{m+1}\bigr)^k\Biggr]\mathfrak{D}_{m+1}^{-1}
\begin{pmatrix}
G(0,z)\\ G(1,z)\\ G(2,z)\\ G(3,z)\\ \vdots\\ G(m-1,z) \\ G(m,z)
\end{pmatrix},
\end{equation}
or equivalently,
\begin{equation}\label{rakaij-mstrix}
\begin{pmatrix}a_{i,j}\end{pmatrix}_{0\le i,j\le m}
=\mathcal{A}_{m+1}\Biggl[I_{m+1}+\sum_{k=1}^{m}(-1)^k\bigl(\mathcal{D}_{m+1}^{-1} \mathcal{L}_{m+1}\bigr)^k\Biggr]\mathcal{D}_{m+1}^{-1}
\end{equation}
and
\begin{equation}\label{akaij-mstrix}
\begin{pmatrix}a_{i,j}\end{pmatrix}_{0\le i,j\le m}
=\mathfrak{A}_{m+1}\Biggl[I_{m+1}+\sum_{k=1}^{m}(-1)^k\bigl(\mathfrak{D}_{m+1}^{-1} \mathfrak{L}_{m+1}\bigr)^k\Biggr]\mathfrak{D}_{m+1}^{-1},
\end{equation}
where the $(m+1)\times(m+1)$ square matrices $\mathcal{D}_{m+1}$, $\mathfrak{D}_{m+1}$, $\mathcal{L}_{m+1}$ and $\mathfrak{L}_{m+1}$ are defined by the diagonal matrices
\begin{equation*}
\mathcal{D}_{m+1}=\mathfrak{D}_{m+1}=\diag\bigl(2^0,2^1,2^2,\dotsc,2^m\bigr)
\end{equation*}
and the strictly lower triangular matrices
\begin{equation*}
\mathcal{L}_{m+1}=\begin{pmatrix}
0 & 0 & 0 & 0 & \dotsm & 0 & 0\\
\beta_{1,0} & 0 & 0 & 0 & \dotsm & 0 & 0\\
\beta_{2,0} & \beta_{2,1} & 0 & 0 & \dotsm & 0 & 0\\
\beta_{3,0} & \beta_{3,1} & \beta_{3,2} & 0 & \dotsm & 0 & 0\\
\vdots & \vdots & \vdots & \vdots & \ddots & \vdots & \vdots\\
\beta_{m-1,0} & \beta_{m-1,1} & \beta_{m-1,2} & \beta_{m-1,3} & \dotsm & 0 & 0\\
\beta_{m,0} & \beta_{m,1} & \beta_{m,2} & \beta_{m,3} & \dotsm & \beta_{m,m-1} & 0
\end{pmatrix}
\end{equation*}
and
\begin{equation*}
\mathfrak{L}_{m+1}=\begin{pmatrix}
0 & 0 & 0 & 0 & \dotsm & 0 & 0\\
\mu_{1,0} & 0 & 0 & 0 & \dotsm & 0 & 0\\
\mu_{2,0} & \mu_{2,1} & 0 & 0 & \dotsm & 0 & 0\\
\mu_{3,0} & \mu_{3,1} & \mu_{3,2} & 0 & \dotsm & 0 & 0\\
\vdots & \vdots & \vdots & \vdots & \ddots & \vdots & \vdots\\
\mu_{m-1,0} & \mu_{m-1,1} & \mu_{m-1,2} & \mu_{m-1,3} & \dotsm & 0 & 0\\
\mu_{m,0} & \mu_{m,1} & \mu_{m,2} & \mu_{m,3} & \dotsm & \mu_{m,m-1} & 0
\end{pmatrix}.
\end{equation*}
\end{thm}

\begin{proof}
Since the matrices $\mathcal{A}_{m+1}$, $\mathcal{B}_{m+1}$, $\mathfrak{A}_{m+1}$, and $\mathfrak{B}_{m+1}$ are lower triangular ones with
\begin{equation*}
\mathcal{B}_{m+1}=\mathcal{D}_{m+1}+\mathcal{L}_{m+1} \quad\text{and}\quad \mathfrak{B}_{m+1}=\mathfrak{D}_{m+1}+\mathfrak{L}_{m+1},
\end{equation*}
applying the formula~\eqref{beta-matrix-inv-D-eq} in Proposition~\ref{Im-Inverse-lem} to the matrix equations~\eqref{FG(m-x)-matrix-Relation} and~\eqref{frak-AB-matrix} in Theorems~\ref{FG(m-x)-matrix-Rel-thm} and~\ref{frak-AB-matrix-thm}, we readily arrive at the matrix equations~\eqref{FG(m-x)-matrix-Final} and~\eqref{FG(m-x)-matrix-frak}.
The proof of Theorem~\ref{FG(m-x)-matrix-Final-Thm} is thus complete.
\end{proof}

\begin{cor}\label{recursion-cor}
For $m\in\mathbb{N}_0$, the coefficients $a_{m,j}$ in~\eqref{Long-Propb} are determined by the backward recursion
\begin{equation}\label{a-recursion}
a_{m,m}=\frac{1}{2^{m+1}} \quad\text{and}\quad
a_{m,j}=\frac{1}{2^{j}}\Biggl[\alpha_{m,j}-\sum_{i=j+1}^{m}a_{m,i}\,\beta_{i,j}\Biggr]
\end{equation}
for $j=m-1,m-2,\dotsc,0$, 
where $\alpha_{m,j}$ and $\beta_{i,j}$ are the entries of $\mathcal{A}_{m+1}$ and $\mathcal{B}_{m+1}$ in~\eqref{Matrix-A(m+1)-dfn} and~\eqref{Matrix-B(m+1)-dfn}.
\end{cor}

\begin{proof}
By the equations~\eqref{F(m-z)-z-polyn}, \eqref{G(m-z)-z-polyn}, and~\eqref{Long-Propb}, we obtain
\begin{equation*}
\sum_{j=0}^{m}\alpha_{m,j}z^j=\sum_{i=0}^{m}a_{m,i}\sum_{j=0}^{i}\beta_{i,j}z^j
=\sum_{j=0}^{m}\Biggl[\sum_{i=j}^{m}a_{m,i}\beta_{i,j}\Biggr]z^j.
\end{equation*}
Comparing the coefficients of $z^j$ gives
\[
\alpha_{m,j}=\sum_{i=j}^{m}a_{m,i}\beta_{i,j}, \quad0\le j\le m.
\]
Since $\beta_{j,j}=2^j$, solving for the term $i=j$ yields~\eqref{a-recursion}, and $j=m$ gives 
\[
a_{m,m}=\frac{\alpha_{m,m}}{2^m}=\frac1{2^{m+1}}.
\]
The proof of Corollary~\ref{recursion-cor} is complete.
\end{proof}

\section{A slight generalization of main results}
The sequence $F(m,n)$ in~\eqref{F(m-n)} can be further extended as
\begin{equation}\label{F(m-n)-gen}
F_{\lambda,\mu}(t,z)=2^t[\zeta(-t,\lambda z)-\zeta(-t,\lambda z+\mu)]
\end{equation}
for $t\in\mathbb{C}\setminus\{-1\}$ and $\Re(z)>0$ with $\lambda,\mu>0$. It is easy to see that
\begin{equation*}
F_{1/2,1/2}(m,n+1)=F(m,n), \quad m,n\in\mathbb{N}_0.
\end{equation*}
\par
Can one express the function $F_{\lambda,\mu}(t,z)$ for $t\in\mathbb{C}\setminus\{-1\}$, $\Re(z)>0$, and $\lambda,\mu>0$ in~\eqref{F(m-n)-gen} in terms of the functional sequence
\begin{equation}\label{G(m-n)-gen}
G(m,z)=\Gamma(m+1){\,}_2F_1(-m,-z;1;2)
\end{equation}
for $m\in\mathbb{N}_0$ and $z\in\mathbb{C}$? We give a partial answer, for $t=m\in\mathbb{N}_0$, in the following theorem. The case of general $t\in\mathbb{C}\setminus\{-1\}$ remains open.

\begin{thm}\label{lambda-mu-thm}
For $m\in\mathbb{N}_0$ and $\lambda,\mu>0$, let
\begin{equation}\label{P-matrix}
P_{m+1}=\begin{pmatrix}p_{i,j}\end{pmatrix}_{0\le i,j\le m}
=\begin{pmatrix}\displaystyle
\frac{2^i\mu^{i+1}}{i+1} \sum_{k=0}^{i-j} \binom{i+1}{k} \binom{i-k+1}{j} \frac{B_k}{\mu^k} \biggl(\frac{\lambda}{\mu}\biggr)^j
\end{pmatrix}_{0\le i,j\le m}
\end{equation}
and
\begin{equation}\label{Q-matrix}
Q_{m+1}=\begin{pmatrix}q_{i,j}\end{pmatrix}_{0\le i,j\le m}
=\begin{pmatrix}\displaystyle
i!\sum_{k=j}^{i}\binom{i}{k}\frac{2^k}{k!}s(k,j)
\end{pmatrix}_{0\le i,j\le m}
\end{equation}
be two $(m+1)\times(m+1)$ matrices, where $B_k$ and $s(k,j)$ stand for the Bernoulli numbers and the signed Stirling numbers of the first kind, respectively. Then the matrix equation
\begin{equation}\label{FG(m-x)-matrix-lambda-mu}
\begin{pmatrix}
F_{\lambda,\mu}(0,z)\\ F_{\lambda,\mu}(1,z)\\ F_{\lambda,\mu}(2,z)\\ F_{\lambda,\mu}(3,z)\\ \vdots\\ F_{\lambda,\mu}(m-1,z) \\ F_{\lambda,\mu}(m,z)
\end{pmatrix}
=
P_{m+1}Q_{m+1}^{-1}
\begin{pmatrix}
G(0,z)\\ G(1,z)\\ G(2,z)\\ G(3,z)\\ \vdots\\ G(m-1,z) \\ G(m,z)
\end{pmatrix}
\end{equation}
is valid for $\Re(z)>0$, where $Q_{m+1}^{-1}$ represents the inverse of the matrix $Q_{m+1}$.
\end{thm}

\begin{proof}[First proof of Theorem~\ref{lambda-mu-thm}]
Making use of the formulas~\eqref{zeta-bernoulli-no} and~\eqref{p.588NIST-HB-2010}, we arrive at
\begin{equation}
\begin{aligned}\label{F(m-x)-x-polyn-lambda-mu}
F_{\lambda,\mu}(m,z)&=\frac{2^m}{m+1}[B_{m+1}(\lambda z+\mu)-B_{m+1}(\lambda z)]\\
&=\frac{2^m}{m+1}\sum_{k=0}^{m+1}\binom{m+1}{k}B_k\bigl[(\lambda z+\mu)^{m+1-k}-(\lambda z)^{m+1-k}\bigr]\\
&=\frac{2^m}{m+1}\sum_{k=0}^{m+1}\binom{m+1}{k}B_k\sum_{\ell=0}^{m-k}\binom{m+1-k}{\ell}(\lambda z)^\ell\mu^{m+1-k-\ell}\\
&=\frac{2^m\mu^{m+1}}{m+1}\sum_{\ell=0}^{m} \Biggl[\sum_{k=0}^{m-\ell} \binom{m+1}{k} \binom{m+1-k}{\ell} \frac{B_k}{\mu^k}\Biggr] \biggl(\frac{\lambda}{\mu}\biggr)^\ell z^\ell\\
&=\sum_{\ell=0}^{m}p_{m,\ell}z^\ell
\end{aligned}
\end{equation}
for $m\in\mathbb{N}_0$ and $\Re(z)>0$ with $\lambda,\mu>0$.
\par
By the definition~\eqref{hypergeom-f} for $a\in\{0,-1,-2,\dotsc\}$ or $b\in\{0,-1,-2,\dotsc\}$, we have
\begin{align*}
{\,}_2F_1(-m,-z;1;2)&=\sum_{k=0}^m\frac{(-m)_k(-z)_k} {k!}\frac{2^k}{k!}\\
&=\sum_{k=0}^m(-1)^{k}\frac{(-m)_k} {k!}\frac{2^k}{k!}\sum_{\ell=0}^{k}s(k,\ell)z^\ell\\
&=\sum_{\ell=0}^{m}\Biggl[\sum_{k=\ell}^{m}\binom{m}{k}\frac{2^k}{k!}s(k,\ell)\Biggr]z^\ell
\end{align*}
for $m\in\mathbb{N}_0$ and $z\in\mathbb{C}$, where we used the relation
\begin{equation*}
(z)_k=\sum_{\ell=0}^{k}(-1)^{k-\ell}s(k,\ell)z^\ell
\end{equation*}
in~\cite[p.~213]{Comtet-Combinatorics-74}.
Therefore, we obtain
\begin{equation}\label{G(m-x)-x-polyn}
G(m,z)=m!\sum_{\ell=0}^{m}\Biggl[\sum_{k=\ell}^{m}\binom{m}{k}\frac{2^k}{k!}s(k,\ell)\Biggr]z^\ell
=\sum_{\ell=0}^{m}q_{m,\ell}z^\ell
\end{equation}
for $m\in\mathbb{N}_0$ and $z\in\mathbb{C}$.
\par
The formulas~\eqref{F(m-x)-x-polyn-lambda-mu} and~\eqref{G(m-x)-x-polyn} can be rearranged as matrix equations
\begin{equation}\label{F-P-Matrix}
\begin{pmatrix}
F_{\lambda,\mu}(0,z)\\ F_{\lambda,\mu}(1,z)\\ F_{\lambda,\mu}(2,z)\\ F_{\lambda,\mu}(3,z)\\ \vdots\\ F_{\lambda,\mu}(m-1,z) \\ F_{\lambda,\mu}(m,z)
\end{pmatrix}
=P_{m+1}\begin{pmatrix}
1\\ z\\ z^2\\ z^3\\ \vdots\\ z^{m-1}\\ z^m
\end{pmatrix}, \quad \Re(z)>0
\end{equation}
and
\begin{equation}\label{G-Q-matrix}
\begin{pmatrix}
G(0,z)\\ G(1,z)\\ G(2,z)\\ G(3,z)\\ \vdots\\ G(m-1,z) \\ G(m,z)
\end{pmatrix}
=Q_{m+1}\begin{pmatrix}
1\\ z\\ z^2\\ z^3\\ \vdots\\ z^{m-1}\\ z^m
\end{pmatrix}, \quad z\in\mathbb{C}
\end{equation}
for $m\in\mathbb{N}_0$ and $\lambda,\mu>0$, respectively.
Because $q_{i,i}=2^i$ for $0\le i\le m\in\mathbb{N}_0$, the $(m+1)\times(m+1)$ square matrix $Q_{m+1}$ is invertible. Consequently, we acquire the matrix equation~\eqref{FG(m-x)-matrix-lambda-mu}. Moreover, since $G(j,\cdot)$ has degree exactly $j$ (because $q_{j,j}=2^j\neq0$), the coefficients are unique by Lemma~\ref{unique-lem} with $p_j=G(j,\cdot)$ and any infinite set $S\subset(0,\infty)$.
The first proof of Theorem~\ref{lambda-mu-thm} is complete.
\end{proof}

\begin{proof}[Second proof of Theorem~\ref{lambda-mu-thm}]
Utilizing~\eqref{zeta-bernoulli-no}, \eqref{p.588NIST-HB-2010}, and the formula~\eqref{p805Entry23.1.11abram}, we arrive at
\begin{align*}
F_{\lambda,\mu}(m,z)&=\frac{2^m}{m+1}[B_{m+1}(\lambda z+\mu)-B_{m+1}(\lambda z)]\\
&=2^m\int_{\lambda z}^{\lambda z+\mu}B_m(t)\td t\\
&=2^m\int_{\lambda z}^{\lambda z+\mu}\sum_{k=0}^{m}\binom{m}{k}B_k t^{m-k}\td t\\
&=2^m\sum_{k=0}^{m}\binom{m}{k}\frac{B_k}{m-k+1} \bigl[(\lambda z+\mu)^{m-k+1}-(\lambda z)^{m-k+1}\bigr]\\
&=2^m\sum_{k=0}^{m}\binom{m}{k}\frac{B_k}{m-k+1} \sum_{\ell=0}^{m-k}\binom{m-k+1}{\ell}(\lambda z)^\ell\mu^{m-k-\ell+1}\\
&=2^m\sum_{\ell=0}^{m}\Biggl[\sum_{k=0}^{m-\ell}\frac{\mu^{m-k+1}}{m-k+1} \binom{m}{k}B_k\binom{m-k+1}{\ell}\Biggr] \biggl(\frac{\lambda}{\mu}\biggr)^\ell z^\ell\\
&=\sum_{\ell=0}^{m}p_{m,\ell} z^\ell
\end{align*}
for $m\in\mathbb{N}_0$ and $\Re(z)>0$. Equivalently, the matrix equation~\eqref{F-P-Matrix} follows.
\par
Combining the matrix equation~\eqref{F-P-Matrix} with~\eqref{G-Q-matrix} leads to~\eqref{FG(m-x)-matrix-lambda-mu}. The second proof of Theorem~\ref{lambda-mu-thm} is complete.
\end{proof}

\section{Remarks}
In this section, we list several remarks on our main results.

\begin{rem}
Taking $\lambda=\mu=\frac12$ in Theorem~\ref{lambda-mu-thm} gives $P_{m+1}=\mathfrak{A}_{m+1}$ and $Q_{m+1}=\mathcal{B}_{m+1}$, where $\mathfrak{A}_{m+1}$ is defined in~\eqref{lambda-matrix} and $\mathcal{B}_{m+1}$ in~\eqref{Matrix-B(m+1)-dfn}. Since $F_{1/2,1/2}(m,z)=F(m,z-1)$, Theorem~\ref{lambda-mu-thm} then yields
\begin{equation*}
\begin{pmatrix}
F(0,z-1)\\ F(1,z-1)\\ \vdots\\ F(m,z-1)
\end{pmatrix}
=\mathfrak{A}_{m+1}\mathcal{B}_{m+1}^{-1}
\begin{pmatrix}
G(0,z)\\ G(1,z)\\ \vdots\\ G(m,z)
\end{pmatrix},
\end{equation*}
which pairs the left factor of Theorem~\ref{frak-AB-matrix-thm} with the right factor of Theorem~\ref{FG(m-x)-matrix-Rel-thm}. This is a shifted relation and is different from both of those theorems.
\end{rem}

\begin{rem}
For more information on Proposition~\ref{Im-Inverse-lem}, please refer to the questions at the websites \url{https://math.stackexchange.com/q/47543} (accessed on 18 January 2025) and \url{https://math.stackexchange.com/q/5024354} (accessed on 18 January 2025) as well as their answers therein.
\end{rem}

\begin{rem}
When $m=9$, by virtue of the software Mathematica~14.0, we acquire
\begin{align*}
\mathcal{A}_{10}&=
\begin{pmatrix}
 \frac{1}{2} & 0 & 0 & 0 & 0 & 0 & 0 & 0 & 0 & 0 \\
 \frac{1}{4} & \frac{1}{2} & 0 & 0 & 0 & 0 & 0 & 0 & 0 & 0 \\
 0 & \frac{1}{2} & \frac{1}{2} & 0 & 0 & 0 & 0 & 0 & 0 & 0 \\
 -\frac{1}{8} & 0 & \frac{3}{4} & \frac{1}{2} & 0 & 0 & 0 & 0 & 0 & 0 \\
 0 & -\frac{1}{2} & 0 & 1 & \frac{1}{2} & 0 & 0 & 0 & 0 & 0 \\
 \frac{1}{4} & 0 & -\frac{5}{4} & 0 & \frac{5}{4} & \frac{1}{2} & 0 & 0 & 0 & 0 \\
 0 & \frac{3}{2} & 0 & -\frac{5}{2} & 0 & \frac{3}{2} & \frac{1}{2} & 0 & 0 & 0 \\
 -\frac{17}{16} & 0 & \frac{21}{4} & 0 & -\frac{35}{8} & 0 & \frac{7}{4} & \frac{1}{2} & 0 & 0 \\
 0 & -\frac{17}{2} & 0 & 14 & 0 & -7 & 0 & 2 & \frac{1}{2} & 0 \\
 \frac{31}{4} & 0 & -\frac{153}{4} & 0 & \frac{63}{2} & 0 & -\frac{21}{2} & 0 & \frac{9}{4} & \frac{1}{2}
\end{pmatrix},\\
\mathcal{B}_{10}^{-1}
&=
\begin{pmatrix}
1 & 0 & 0 & 0 & 0 & 0 & 0 & 0 & 0 & 0\\
1 & 2 & 0 & 0 & 0 & 0 & 0 & 0 & 0 & 0\\
2 & 4 & 4 & 0 & 0 & 0 & 0 & 0 & 0 & 0\\
6 & 16 & 12 & 8 & 0 & 0 & 0 & 0 & 0 & 0\\
24 & 64 & 80 & 32 & 16 & 0 & 0 & 0 & 0 & 0\\
120 & 368 & 400 & 320 & 80 & 32 & 0 & 0 & 0 & 0\\
720 & 2208 & 3136 & 1920 & 1120 & 192 & 64 & 0 & 0 & 0\\
5040 & 16896 & 21952 & 19712 & 7840 & 3584 & 448 & 128 & 0 & 0\\
40320 & 135168 & 209408 & 157696 & 102144 & 28672 & 10752 & 1024 & 256 & 0\\
362880 & 1297152 & 1884672 & 1838080 & 919296 & 462336 & 96768 & 30720 & 2304 & 512
\end{pmatrix}^{-1}\\
&=\begin{pmatrix}
 1 & 0 & 0 & 0 & 0 & 0 & 0 & 0 & 0 & 0 \\
 -\frac{1}{2} & \frac{1}{2} & 0 & 0 & 0 & 0 & 0 & 0 & 0 & 0 \\
 0 & -\frac{1}{2} & \frac{1}{4} & 0 & 0 & 0 & 0 & 0 & 0 & 0 \\
 \frac{1}{4} & -\frac{1}{4} & -\frac{3}{8} & \frac{1}{8} & 0 & 0 & 0 & 0 & 0 & 0 \\
 0 & 1 & -\frac{1}{2} & -\frac{1}{4} & \frac{1}{16} & 0 & 0 & 0 & 0 & 0 \\
 -\frac{1}{2} & \frac{1}{2} & \frac{15}{8} & -\frac{5}{8} & -\frac{5}{32} & \frac{1}{32} & 0 & 0 & 0 & 0 \\
 0 & -\frac{17}{4} & \frac{17}{8} & \frac{5}{2} & -\frac{5}{8} & -\frac{3}{32} & \frac{1}{64} & 0 & 0 & 0 \\
 \frac{17}{8} & -\frac{51}{64} & -\frac{1933}{128} & \frac{129}{32} & \frac{375}{128} & -\frac{265}{512} & -\frac{61}{1024} & \frac{1}{128} & 0 & 0 \\
 0 & \frac{411}{16} & -\frac{411}{32} & -\frac{227}{8} & \frac{227}{32} & \frac{321}{128} & -\frac{107}{256} & -\frac{1}{32} & \frac{1}{256} & 0 \\
 -\frac{31}{2} & -\frac{1289}{32} & \frac{12345}{64} & -\frac{355}{16} & -\frac{3885}{64} & \frac{2373}{256} & \frac{1281}{512} & -\frac{21}{64} & -\frac{9}{512} & \frac{1}{512} \\
\end{pmatrix},
\end{align*}
and
\begin{equation*}
\mathcal{A}_{10}\mathcal{B}_{10}^{-1}=
\begin{pmatrix}
 \frac{1}{2} & 0 & 0 & 0 & 0 & 0 & 0 & 0 & 0 & 0 \\
 0 & \frac{1}{4} & 0 & 0 & 0 & 0 & 0 & 0 & 0 & 0 \\
 -\frac{1}{4} & 0 & \frac{1}{8} & 0 & 0 & 0 & 0 & 0 & 0 & 0 \\
 0 & -\frac{1}{2} & 0 & \frac{1}{16} & 0 & 0 & 0 & 0 & 0 & 0 \\
 \frac{1}{2} & 0 & -\frac{5}{8} & 0 & \frac{1}{32} & 0 & 0 & 0 & 0 & 0 \\
 0 & \frac{17}{8} & 0 & -\frac{5}{8} & 0 & \frac{1}{64} & 0 & 0 & 0 & 0 \\
 -\frac{17}{8} & 0 & \frac{77}{16} & 0 & -\frac{35}{64} & 0 & \frac{1}{128} & 0 & 0 & 0 \\
 0 & -\frac{31}{2} & 0 & \frac{63}{8} & 0 & -\frac{7}{16} & 0 & \frac{1}{256} & 0 & 0 \\
 \frac{31}{2} & 0 & -55 & 0 & \frac{21}{2} & 0 & -\frac{21}{64} & 0 & \frac{1}{512} & 0 \\
 0 & \frac{691}{4} & 0 & -\frac{265}{2} & 0 & \frac{777}{64} & 0 & -\frac{15}{64} & 0 & \frac{1}{1024}
\end{pmatrix}.
\end{equation*}
See also the question at the site \url{https://math.stackexchange.com/q/5024354} (accessed on 18 January 2025).
From the equation~\eqref{FG(m-x)-matrix-Relation} in Theorem~\ref{FG(m-x)-matrix-Rel-thm}, it follows that
\begin{equation}\label{FG(m-n)-matrix-Rel10}
\begin{pmatrix}
F(0,n)\\ F(1,n)\\ F(2,n)\\ F(3,n)\\ F(4,n)\\ F(5,n)\\ F(6,n)\\ F(7,n)\\ F(8,n) \\ F(9,n)
\end{pmatrix}
=
\begin{pmatrix}
 \frac{1}{2} & 0 & 0 & 0 & 0 & 0 & 0 & 0 & 0 & 0 \\
 0 & \frac{1}{4} & 0 & 0 & 0 & 0 & 0 & 0 & 0 & 0 \\
 -\frac{1}{4} & 0 & \frac{1}{8} & 0 & 0 & 0 & 0 & 0 & 0 & 0 \\
 0 & -\frac{1}{2} & 0 & \frac{1}{16} & 0 & 0 & 0 & 0 & 0 & 0 \\
 \frac{1}{2} & 0 & -\frac{5}{8} & 0 & \frac{1}{32} & 0 & 0 & 0 & 0 & 0 \\
 0 & \frac{17}{8} & 0 & -\frac{5}{8} & 0 & \frac{1}{64} & 0 & 0 & 0 & 0 \\
 -\frac{17}{8} & 0 & \frac{77}{16} & 0 & -\frac{35}{64} & 0 & \frac{1}{128} & 0 & 0 & 0 \\
 0 & -\frac{31}{2} & 0 & \frac{63}{8} & 0 & -\frac{7}{16} & 0 & \frac{1}{256} & 0 & 0 \\
 \frac{31}{2} & 0 & -55 & 0 & \frac{21}{2} & 0 & -\frac{21}{64} & 0 & \frac{1}{512} & 0 \\
 0 & \frac{691}{4} & 0 & -\frac{265}{2} & 0 & \frac{777}{64} & 0 & -\frac{15}{64} & 0 & \frac{1}{1024}
\end{pmatrix}
\begin{pmatrix}
G(0,n)\\ G(1,n)\\ G(2,n)\\ G(3,n)\\ G(4,n)\\ G(5,n)\\ G(6,n)\\ G(7,n)\\ G(8,n) \\ G(9,n)
\end{pmatrix}.
\end{equation}
\par
The matrix equation~\eqref{FG(m-n)-matrix-Rel10} can also be rewritten as
\begin{align*}
F(0,n) &=\frac{1}{2}G(0,n),\\
F(1,n) &=\frac{1}{4}G(1,n),\\
F(2,n) &=\frac{1}{8}G(2,n)-\frac{1}{4}G(0,n),\\
F(3,n) &=\frac{1}{16}G(3,n)-\frac{1}{2}G(1,n),\\
F(4,n) &=\frac{1}{32}G(4,n)-\frac{5}{8}G(2,n)+\frac{1}{2}G(0,n),\\
F(5,n) &=\frac{1}{64}G(5,n)-\frac{5}{8}G(3,n)+\frac{17}{8}G(1,n),\\
F(6,n) &=\frac{1}{128}G(6,n)-\frac{35}{64}G(4,n)+\frac{77}{16}G(2,n)-\frac{17}{8}G(0,n),\\
F(7,n) &=\frac{1}{256}G(7,n)-\frac{7}{16}G(5,n)+\frac{63}{8}G(3,n)-\frac{31}{2}G(1,n),\\
F(8,n) &=\frac{1}{512}G(8,n)-\frac{21}{64}G(6,n)+\frac{21}{2}G(4,n)-55G(2,n)+\frac{31}{2}G(0,n),\\
F(9,n) &=\frac{1}{1024}G(9,n)-\frac{15}{64}G(7,n)+\frac{777}{64}G(5,n)-\frac{265}{2}G(3,n)+\frac{691}{4}G(1,n).
\end{align*}
\end{rem}

\begin{rem}
When $m=9$, with the aid of the software Mathematica~14.0, direct computation gives
\begin{align*}
\mathfrak{A}_{10}&=
\begin{pmatrix}
 \frac{1}{2} & 0 & 0 & 0 & 0 & 0 & 0 & 0 & 0 & 0 \\
 -\frac{1}{4} & \frac{1}{2} & 0 & 0 & 0 & 0 & 0 & 0 & 0 & 0 \\
 0 & -\frac{1}{2} & \frac{1}{2} & 0 & 0 & 0 & 0 & 0 & 0 & 0 \\
 \frac{1}{8} & 0 & -\frac{3}{4} & \frac{1}{2} & 0 & 0 & 0 & 0 & 0 & 0 \\
 0 & \frac{1}{2} & 0 & -1 & \frac{1}{2} & 0 & 0 & 0 & 0 & 0 \\
 -\frac{1}{4} & 0 & \frac{5}{4} & 0 & -\frac{5}{4} & \frac{1}{2} & 0 & 0 & 0 & 0 \\
 0 & -\frac{3}{2} & 0 & \frac{5}{2} & 0 & -\frac{3}{2} & \frac{1}{2} & 0 & 0 & 0 \\
 \frac{17}{16} & 0 & -\frac{21}{4} & 0 & \frac{35}{8} & 0 & -\frac{7}{4} & \frac{1}{2} & 0 & 0 \\
 0 & \frac{17}{2} & 0 & -14 & 0 & 7 & 0 & -2 & \frac{1}{2} & 0 \\
 -\frac{31}{4} & 0 & \frac{153}{4} & 0 & -\frac{63}{2} & 0 & \frac{21}{2} & 0 & -\frac{9}{4} & \frac{1}{2} \\
\end{pmatrix},\\
\mathfrak{B}_{10}^{-1}&=
\begin{pmatrix}
 1 & 0 & 0 & 0 & 0 & 0 & 0 & 0 & 0 & 0 \\
 -1 & 2 & 0 & 0 & 0 & 0 & 0 & 0 & 0 & 0 \\
 2 & -4 & 4 & 0 & 0 & 0 & 0 & 0 & 0 & 0 \\
 -6 & 16 & -12 & 8 & 0 & 0 & 0 & 0 & 0 & 0 \\
 24 & -64 & 80 & -32 & 16 & 0 & 0 & 0 & 0 & 0 \\
 -120 & 368 & -400 & 320 & -80 & 32 & 0 & 0 & 0 & 0 \\
 720 & -2208 & 3136 & -1920 & 1120 & -192 & 64 & 0 & 0 & 0 \\
 -5040 & 16896 & -21952 & 19712 & -7840 & 3584 & -448 & 128 & 0 & 0 \\
 40320 & -135168 & 209408 & -157696 & 102144 & -28672 & 10752 & -1024 & 256 & 0 \\
 -362880 & 1297152 & -1884672 & 1838080 & -919296 & 462336 & -96768 & 30720 & -2304 & 512 \\
\end{pmatrix}^{-1}\\
&=\begin{pmatrix}
 1 & 0 & 0 & 0 & 0 & 0 & 0 & 0 & 0 & 0 \\
 \frac{1}{2} & \frac{1}{2} & 0 & 0 & 0 & 0 & 0 & 0 & 0 & 0 \\
 0 & \frac{1}{2} & \frac{1}{4} & 0 & 0 & 0 & 0 & 0 & 0 & 0 \\
 -\frac{1}{4} & -\frac{1}{4} & \frac{3}{8} & \frac{1}{8} & 0 & 0 & 0 & 0 & 0 & 0 \\
 0 & -1 & -\frac{1}{2} & \frac{1}{4} & \frac{1}{16} & 0 & 0 & 0 & 0 & 0 \\
 \frac{1}{2} & \frac{1}{2} & -\frac{15}{8} & -\frac{5}{8} & \frac{5}{32} & \frac{1}{32} & 0 & 0 & 0 & 0 \\
 0 & \frac{17}{4} & \frac{17}{8} & -\frac{5}{2} & -\frac{5}{8} & \frac{3}{32} & \frac{1}{64} & 0 & 0 & 0 \\
 -\frac{17}{8} & -\frac{17}{8} & \frac{231}{16} & \frac{77}{16} & -\frac{175}{64} & -\frac{35}{64} & \frac{7}{128} & \frac{1}{128} & 0 & 0 \\
 0 & -31 & -\frac{31}{2} & \frac{63}{2} & \frac{63}{8} & -\frac{21}{8} & -\frac{7}{16} & \frac{1}{32} & \frac{1}{256} & 0 \\
 \frac{31}{2} & \frac{31}{2} & -165 & -55 & \frac{105}{2} & \frac{21}{2} & -\frac{147}{64} & -\frac{21}{64} & \frac{9}{512} & \frac{1}{512} \\
\end{pmatrix},
\end{align*}
and
\begin{equation*}
\mathfrak{A}_{10}\mathfrak{B}_{10}^{-1}=
\begin{pmatrix}
 \frac{1}{2} & 0 & 0 & 0 & 0 & 0 & 0 & 0 & 0 & 0 \\
 0 & \frac{1}{4} & 0 & 0 & 0 & 0 & 0 & 0 & 0 & 0 \\
 -\frac{1}{4} & 0 & \frac{1}{8} & 0 & 0 & 0 & 0 & 0 & 0 & 0 \\
 0 & -\frac{1}{2} & 0 & \frac{1}{16} & 0 & 0 & 0 & 0 & 0 & 0 \\
 \frac{1}{2} & 0 & -\frac{5}{8} & 0 & \frac{1}{32} & 0 & 0 & 0 & 0 & 0 \\
 0 & \frac{17}{8} & 0 & -\frac{5}{8} & 0 & \frac{1}{64} & 0 & 0 & 0 & 0 \\
 -\frac{17}{8} & 0 & \frac{77}{16} & 0 & -\frac{35}{64} & 0 & \frac{1}{128} & 0 & 0 & 0 \\
 0 & -\frac{31}{2} & 0 & \frac{63}{8} & 0 & -\frac{7}{16} & 0 & \frac{1}{256} & 0 & 0 \\
 \frac{31}{2} & 0 & -55 & 0 & \frac{21}{2} & 0 & -\frac{21}{64} & 0 & \frac{1}{512} & 0 \\
 0 & \frac{691}{4} & 0 & -\frac{265}{2} & 0 & \frac{777}{64} & 0 & -\frac{15}{64} & 0 & \frac{1}{1024} \\
\end{pmatrix}.
\end{equation*}
It is easy to see that
\begin{equation*}
\mathcal{A}_{10}\ne\mathfrak{A}_{10}, \quad \mathcal{B}_{10}\ne\mathfrak{B}_{10},\quad \mathcal{A}_{10}\mathcal{B}_{10}^{-1}=\mathfrak{A}_{10}\mathfrak{B}_{10}^{-1}.
\end{equation*}
In general, $\mathcal{A}_{m+1}\mathcal{B}_{m+1}^{-1}=\mathfrak{A}_{m+1}\mathfrak{B}_{m+1}^{-1}$ for every $m\in\mathbb{N}_0$.
Indeed, Theorems~\ref{FG(m-x)-matrix-Rel-thm} and~\ref{frak-AB-matrix-thm} give two representations of $F(i,z)$ as a linear combination of $G(0,z),\dotsc,G(m,z)$ for $0\le i\le m$ and $z\in(-1,\infty)$, and these representations must have the same coefficients by Lemma~\ref{unique-lem}.
On the other hand, for $m\in\mathbb{N}$ we have $\mathcal{A}_{m+1}\ne\mathfrak{A}_{m+1}$ and $\mathcal{B}_{m+1}\ne\mathfrak{B}_{m+1}$. Since each of these matrices contains the corresponding matrix for smaller $m$ as a leading principal submatrix, it suffices to check $m=1$, where $\alpha_{1,0}=\tfrac14\ne-\tfrac14=\lambda_{1,0}$ and $\beta_{1,0}=1\neq-1=\mu_{1,0}$.
\end{rem}

\begin{rem}
Since $s(k,j)=0$ and $\binom{k}{j}=0$ for $j>k\in\mathbb{N}_0$, the matrices $\mathcal{A}_{m+1}$ and $\mathcal{B}_{m+1}$ defined in~\eqref{Matrix-A(m+1)-dfn} and~\eqref{Matrix-B(m+1)-dfn} are lower triangular ones. By related knowledge in matrix analysis, we see that the inverse $\mathcal{B}_{m+1}^{-1}$ is also a lower triangular matrix, and then the product $\mathcal{A}_{m+1}\mathcal{B}_{m+1}^{-1}$ is a lower triangular matrix too. Consequently, the entries $a_{i,j}=0$ for $0\le i<j$ is true.
\par
From~\eqref{Matrix-A(m+1)-dfn} and~\eqref{Matrix-B(m+1)-dfn}, we derive
\begin{equation*}
\alpha_{i,i}=\frac{1}{2}
\quad\text{and}\quad
\beta_{i,i}=2^i.
\end{equation*}
From~\eqref{lambda-matrix} and~\eqref{mu-matrix}, we deduce
\begin{equation*}
\lambda_{i,i}=\frac1{2}
\quad\text{and}\quad
\mu_{i,i}=2^i.
\end{equation*}
These mean that the entries $a_{i,i}=\frac{1}{2^{i+1}}$ for $0\le i\le m$.
\par
The matrix equation~\eqref{FG(m-n)-matrix-Rel10} motivates us to guess that the entries $a_{i,j}$ in either~\eqref{aij-mstrix} or~\eqref{frakaij-mstrix} or~\eqref{rakaij-mstrix} or~\eqref{akaij-mstrix} follow the sign pattern below for $k\in\mathbb{N}$:
\begin{enumerate}
\item
$a_{i,j}=0$ if $i-j=2k-1$,
\item
$a_{i,j}<0$ if $i-j=4k-2$, and
\item
$a_{i,j}>0$ if $i-j=4k$.
\end{enumerate}
\end{rem}

\begin{rem}\label{open-aij-rem}
What are the general explicit expressions for the entries $a_{i,j}$ with $0\le j<i\le m$ of the square matrix $\begin{pmatrix}a_{i,j}\end{pmatrix}_{0\le i,j\le m}$ in either~\eqref{aij-mstrix} or~\eqref{frakaij-mstrix} or~\eqref{rakaij-mstrix} or~\eqref{akaij-mstrix}?
\end{rem}

\begin{rem}
When $z=0$, we have
\begin{equation*}
F(m,0)=\eta(-m) \quad \text{and}\quad G(m,0)=\Gamma(m+1)=m!
\end{equation*}
for $m\in\mathbb{N}_0$.
Then the linear combination~\eqref{Long-Propb} becomes
\begin{equation}\label{Long-Propb=0}
\eta(-m)=\sum_{j=0}^m a_{m,j}j!
\end{equation}
for $m\in\mathbb{N}_0$
and the matrix equation~\eqref{FG(m-n)-matrix-Rel10} becomes
\begin{align*}
\eta(0) &=\frac{1}{2}\times0!,\\ 
\eta(-1)&=\frac{1}{4}\times1!,\\
\eta(-2) &=\frac{1}{8}\times2!-\frac{1}{4}\times0!,\\
\eta(-3) &=\frac{1}{16}\times3!-\frac{1}{2}\times1!,\\
\eta(-4) &=\frac{1}{32}\times4!-\frac{5}{8}\times2!+\frac{1}{2}\times0!,\\
\eta(-5) &=\frac{1}{64}\times5!-\frac{5}{8}\times3!+\frac{17}{8}\times1!,\\
\eta(-6) &=\frac{1}{128}\times6!-\frac{35}{64}\times4!+\frac{77}{16}\times2!-\frac{17}{8}\times0!,\\
\eta(-7) &=\frac{1}{256}\times7!-\frac{7}{16}\times5!+\frac{63}{8}\times3!-\frac{31}{2}\times1!,\\
\eta(-8) &=\frac{1}{512}\times8!-\frac{21}{64}\times6!+\frac{21}{2}\times4!-55\times2!+\frac{31}{2}\times0!,\\
\eta(-9) &=\frac{1}{1024}\times9!-\frac{15}{64}\times7!+\frac{777}{64}\times5!-\frac{265}{2}\times3!+\frac{691}{4}\times1!.
\end{align*}
\par
In~\cite[Remark~2]{CMP6528.tex}, see also the papers~\cite{Gottfried-Qi.tex, log-secant-norm-tail.tex} and the question at the site \url{https://math.stackexchange.com/q/307274} (accessed on 19 January 2025) as well as several answers there, the authors derived
\begin{equation}\label{Remark2CMP6528.tex}
\eta(1-n)=\sum_{k=1}^n(-1)^{k-1}\frac{(k-1)!}{2^{k}}S(n,k),\quad n\in\mathbb{N},
\end{equation}
with $\eta(-2n)=0$ for $n\in\mathbb{N}$ and
\begin{equation}\label{Remark2CMP6528.tex2}
\eta(1-2n)=\sum_{k=1}^{2n}(-1)^{k-1}\frac{(k-1)!}{2^{k}}S(2n,k),\quad n\in\mathbb{N}.
\end{equation}
\par
The formula~\eqref{Remark2CMP6528.tex} can be reformulated as
\begin{equation}\label{Remark2CMP6528.tex-rew}
\eta(-m)=\sum_{j=0}^{m}\frac{(-1)^{j}}{2^{j+1}}S(m+1,j+1)j!,\quad m\in\mathbb{N}_0.
\end{equation}
Comparing~\eqref{Long-Propb=0} with~\eqref{Remark2CMP6528.tex-rew} reveals that
\begin{equation*}
\begin{pmatrix}a_{i,j}\end{pmatrix}_{0\le i,j\le m}\ne \begin{pmatrix}\displaystyle \frac{(-1)^{j}}{2^{j+1}}S(i+1,j+1)\end{pmatrix}_{0\le i,j\le m}, \quad m\in\mathbb{N}.
\end{equation*}
\par
By the way, the formulas~\eqref{Remark2CMP6528.tex} and~\eqref{Remark2CMP6528.tex2} correct several minor typos in~\cite[Remark~2]{CMP6528.tex}.
\end{rem}

\begin{rem}
It is clear that
\begin{equation*}
F(t,n)=-2^t\int_{(1+n)/2}^{(2+n)/2}\frac{\partial \zeta(-t,v)}{\partial v}\td v.
\end{equation*}
Making use of the partial derivative
\begin{equation*}
\frac{\partial \zeta(s,a)}{\partial a}=-s\zeta(s+1,a)
\end{equation*}
for $s\ne0,1$ and $\Re(a)>0$, see~\cite[p.~608, Entry~25.11.17]{NIST-HB-2010}, we obtain
\begin{equation*}
F(t,n)=-2^t t\int_{(1+n)/2}^{(2+n)/2}\zeta(1-t,v)\td v, \quad t\in(0,\infty).
\end{equation*}
\par
The sequences $F(m,n)$ and $G(m,n)$ can also be extended to $F(t,n)$ and $G(t,n)$ for $t\in(-1,\infty)$ and $n\in\mathbb{N}_0$.
Can one find out the scalars $b_{n,j}$ such that
\begin{equation*}
F(t,n)=\sum_{j=0}^{\infty}b_{n,j}G(t,j)
\end{equation*}
is valid for $n\in\mathbb{N}_0$ and $t\in(-1,\infty)$?
\end{rem}

\begin{rem}
For $m\in\mathbb{N}_0$, the matrices $P_{m+1}$ and $Q_{m+1}$ in~\eqref{P-matrix} and~\eqref{Q-matrix} are lower triangular ones.
\end{rem}

\begin{rem}
The function $G(m,z)$ defined by~\eqref{G(m-n)-gen} for $m\in\mathbb{N}_0$ and $z\in\mathbb{C}$ can also be extended to
\begin{equation*}
G(s,z)=\Gamma(s+1){\,}_2F_1(-s,-z;1;2)
\end{equation*}
for $s,z\in\mathbb{C}$. Is there a sequence $s_j\in\mathbb{C}$ for $j\in\mathbb{N}$ such that the infinite series
\begin{equation*}
F_{\lambda,\mu}(t,z)=\sum_{j=1}^{\infty}b_{\lambda,\mu,t,j} G(s_j,z)
\end{equation*}
is valid?
\end{rem}

\begin{rem}
This paper is a revised version of the arXiv electronic preprint~\cite{Long-Uta-P.tex}.
\end{rem}

\section{Declarations}

\paragraph{\bf Acknowledgements}
The author is thankful to Xuan-Yu Long (University of Utah) for his raising the problem in this paper and helpful discussions.

\paragraph*{\bf Author's Contributions}
The author contributed to the whole manuscript and read and approved the final manuscript.

\paragraph*{\bf Funding}
The author was partially supported by the Natural Science Foundation of Inner Mongolia Autonomous Region (Grant No.~2025QN01041) and by the Youth Project of Hulunbuir City for Basic Research and Applied Basic Research (Grant No.~GH2024020).

\paragraph*{\bf Institutional Review Board Statement}
Not applicable.

\paragraph*{\bf Informed Consent Statement}
Not applicable.

\paragraph*{\bf Ethical Approval}
The conducted research does not involve human or animal subjects.

\paragraph*{\bf Availability of Data and Material}
Data sharing is not applicable to this article as no new data were created or analyzed in this study.

\paragraph*{\bf Competing Interests}
The author declares no competing interests.

\paragraph*{\bf Use of AI Tools Declaration}
The author declares that no AI tools were used in the creation of this article.

\end{document}